\documentclass[12 pt]{article}
\usepackage[pdftex]{graphicx}
\usepackage{amsmath}
\usepackage{mathtools}
\usepackage{amsthm}
\usepackage{amsfonts}
\usepackage{amssymb}
\usepackage[margin=1.0in]{geometry}
\usepackage{tikz-cd}
\usetikzlibrary{cd}
\usepackage{enumitem}

\newtheorem{theorem}{Theorem}[section]

\newtheorem{lemma}[theorem]{Lemma}
\newtheorem{prop}[theorem]{Proposition}

\theoremstyle{definition}
\newtheorem{defin}[theorem]{Definition}
\newtheorem{fact}[theorem]{Fact}

\newtheorem{que}[theorem]{Question}

\theoremstyle{remark}
\newtheorem*{rem}{Remark}

\newcommand{\N}{\mathbb{N}}

\renewcommand{\S}{\mathrm{S}}
\newcommand{\M}{\mathrm{M}}
\newcommand{\flim}[1]{\mathrm{Flim}(#1)}

\newcommand{\fr}{Fra\"iss\'e }
\renewcommand{\phi}{\varphi}

\newcommand{\emb}[1]{\mathrm{Emb}(#1)}
\newcommand{\aut}[1]{\mathrm{Aut}(#1)}

\newcommand{\im}[1]{\mathrm{Im}(#1)}

\begin{document}
	\title{Topological dynamics of Polish group extensions}
	\author{Colin Jahel and Andy Zucker}
	\date{June 2020}
	\maketitle
	
	\begin{abstract}
		We consider a short exact sequence $1\to H\to G\to K\to 1$ of Polish groups and consider what can be deduced about the dynamics of $G$ given information about the dynamics of $H$ and $K$. We prove that if the respective universal minimal flows $\M(H)$ and $\M(K)$ are metrizable, then so is $\M(G)$. Furthermore, we show that if $\M(H)$ and $\M(K)$ are metrizable and both $H$ and $K$ are uniquely ergodic, then so is $G$. We then discuss several examples of these phenomena.
		\let\thefootnote\relax
		\let\thefootnote\relax\footnote{The second author was supported by NSF Grant no.\ DMS 1803489.}
	\end{abstract}
	
\section{Introduction}

Let $G$ be a Polish group, and suppose $H\subseteq G$ is a closed, normal subgroup. Setting $K = G/H$, we have that $K$ is also a Polish group, and the quotient map $\pi\colon G\to K$ is a continuous, open homomorphism. In this setting, we say that $G$ is an \emph{extension of $K$ by $H$}. This is the same as saying that 
\begin{equation*}
1\to H\to G\xrightarrow{\pi} K\to 1
\end{equation*}
is a short exact sequence. Examples of group extensions include group products, semidirect products but also more complicated ones as we illustrate in section $5$.

We are interested in dynamical properties of $H$ and $K$ that pass to $G$. In order to describe said properties, we need to introduce a bit of vocabulary. A \emph{$G$-flow} is a continuous (right) action $X\times G\to X$, where $X$ is a compact Hausdorff space. If $X$ and $Y$ are $G$-flows, a \emph{$G$-map} is a continuous map $\phi\colon X\to Y$ which respects the $G$-actions. A $G$-flow is \emph{minimal} if every orbit is dense, or equivalently if it admits no proper subflow.

A well-known result states that every topological group $G$ admits a \emph{universal minimal flow}, a minimal flow which admits a $G$-map onto any other minimal flow. The universal minimal flow is unique up to isomorphism, and is denoted $\M(G)$. This flow can be used to describe several dynamical properties of $G$. For instance, $G$ is \emph{amenable} when $\M(G)$ admits an invariant measure, and $G$ is \emph{uniquely ergodic} if $\M(G)$ admits a unique invariant measure (see \cite{AKL}). The topological group $G$ is called \emph{extremely amenable} if its universal minimal flow is just a point, and for several Polish groups, $\M(G)$ is non-trivial, but still metrizable (see \cite{KPT} for several examples of this). If $\M(G)$ is metrizable, then every minimal $G$-flow is metrizable, and  Ben Yaacov, Melleray and Tsankov prove in \cite{BYMT} that $\M(G)$ has a comeager orbit. We note that all the instances of uniquely ergodic Polish groups that we know have metrizable universal minimal flow.

Our aim in this paper is to describe $\M(G)$ using information about $\M(H)$ and $\M(K)$. In particular, knowing that $\M(H)$ and $\M(K)$ have nice properties, we would like to show that $\M(G)$ also shares these properties. The first theorem shows that metrizability of the universal minimal flow is preserved under group extension and also elaborates on the interaction between $\M(G)$ and $\M(K)$. Notice that $\M(K)$ is a minimal $G$-flow under the action $x\cdot g := x\cdot \pi(g)$, so there is a $G$-map from $\M(G)$ to $\M(K)$. We also denote this $G$-map by $\pi$ for reasons to be explained in Section~\ref{Sec:Background}. 
\vspace{2 mm}

\begin{theorem}
	\label{Thm:MGMetrizable}
	Let $1\to H\to G\xrightarrow{\pi} K\to 1$ be a short exact sequence of Polish groups. If $\M(H)$ and $\M(K)$ are metrizable, then so is $\M(G)$. Furthermore, letting $\pi\colon \M(G)\to \M(K)$ be the canonical map, we have that $\pi^{-1}(\{y\})$ is a minimal $H$-flow for every $y\in \M(K)$.
\end{theorem}
\vspace{2 mm}

Using this description of $\M(G)$, we are also able to prove that when the universal minimal flows are metrizable, then unique ergodicity is stable under group extension.
\vspace{2 mm}

\begin{theorem}
	\label{Thm:GUniqueErg}
	With $H$, $G$, and $K$ as in Theorem~\ref{Thm:MGMetrizable}, then if both $H$ and $K$ are uniquely ergodic, then $G$ is also uniquely ergodic.
\end{theorem}
\vspace{2 mm}

We briefly discuss the organization of the paper. Section~\ref{Sec:Background} gives background on the Samuel compactification of a topological group and the universal minimal flow. Section~\ref{Sec:APPoints} proves the main technical lemma regarding the \emph{almost periodic points} of a $G$-flow. Section~\ref{Sec:AbstractProof} proves the two main theorems. Section~\ref{Sec:Examples} provides examples of Polish group extensions. Section~\ref{Sec:ComboProof} provides a more combinatorial proof of part of Theorem~\ref{Thm:MGMetrizable} by using the connections between topological dynamics and Ramsey theory. Finally, Section~\ref{Sec:Questions} collects some open questions inspired by our work.
\vspace{2 mm}

\emph{Aknowledgements:} The authors thank Lionel Nguyen Van Th\'e and Todor Tsankov for their advice and comments during the writing of this paper.
\vspace{2 mm}

\section{Background}
\label{Sec:Background}





For this section, let $G$ be any topological group. A \emph{$G$-ambit} is a pair $(X, x_0)$ with $X$ a $G$-flow and $x_0\in X$ a distinguished point with dense orbit. If $(Y, y_0)$ is another ambit, we say that $\phi\colon X\to Y$ is a \emph{map of ambits} if $\phi$ is a $G$-map and $\phi(x_0) = y_0$. Notice that there is at most one map of ambits from $(X, x_0)$ to $(Y, y_0)$. One can in fact construct the \emph{greatest ambit}, denoted $(\S(G), 1_G)$, which is an ambit admitting a map of ambits onto any other ambit and is unique up to isomorphism. The orbit $1_G\cdot G$ is homeomorphic to (and identified with) $G$, and $\S(G)$ is often called the \emph{Samuel compactification} of $G$. As an example, when $G$ is a discrete group, then $\S(G)\cong \beta G$, the space of ultrafilters on $G$. In general, $\S(G)$ has the following universal property: if $X$ is a compact space and $f\colon G\to X$ is left-uniformly continuous, then $f$ can be continuously extended to $\S(G)$. For two different constructions of $\S(G)$, see \cite{KPT} or \cite{ZucThesis}.

The universal property of $\S(G)$ allows us to give $\S(G)$ the structure of a \emph{compact left-topological semigroup}, a compact space $S$ endowed with a semigroup structure so that for each $s\in S$, the map $\lambda_s\colon S\to S$ given by $\lambda_s(t) = st$ is continuous. Fix $p\in S(G)$. Then $(\overline{p\cdot G}, p)$ is an ambit, so there is a unique map of ambits $\lambda_p\colon (\S(G), 1_G)\to (\overline{p\cdot G}, p)$. Now given $p, q\in \S(G)$, we declare that $p\cdot q = \lambda_p(q)$. Associativity follows because $\lambda_p\circ \lambda_q$ and $\lambda_{pq}$ are both $G$-maps sending $1_G$ to $pq$, hence they must be equal.

More generally, let $X$ be a $G$-flow. Given $x\in X$, then $(\overline{x\cdot G}, x)$ is an ambit, so there is a unique map of ambits $\lambda_x\colon \S(G)\to X$. Given $p\in \S(G)$, we often write $x\cdot p := \lambda_x(p)$. Notice that if $p, q\in \S(G)$, then $x(pq) = (xp)q$, so the semigroup $\S(G)$ acts on $X$ in a manner which extends the $G$-action.

For a more detailed account of the theory of compact left-topological semigroups, Chapters 1 and 2 of \cite{HinStr} are a great reference (but note the left-right switch between that reference and the presentation here). We will need the following facts, all of which can be found there. Fix a compact left-topological semigroup $S$.

\begin{enumerate}
	\item 
	Every compact left-topological semigroup $S$ contains an \emph{idempotent}, an element $u\in S$ with $u\cdot u = u$.
	\item 
	A \emph{right ideal} is any $I\subseteq S$ for which $I\cdot s \subseteq I$ for every $s\in S$. Notice that if $p\in I$, then $p\cdot S\subseteq I$ is a closed right ideal. It follows that every right ideal contains a minimal right ideal which must be closed. 
	\item 
	If $I\subseteq S$ is a minimal right ideal, then $I$ is a compact left-topological semigroup in its own right, so contains an idempotent. If $u\in I$ is an idempotent, then $uI = I$, so $up = p$ for every $p\in I$. 
	\item
	If $I\subseteq S$ is a minimal right ideal and $p\in I$, then $S\cdot p$ is a minimal \emph{left} ideal, and $pS\cap Sp = Ip$, which is a group. If $u\in Ip$ is the identity of this group, then $Ip = Iu$. So for every $p\in I$, there is an idempotent $u\in I$ with $p\in Iu$.
\end{enumerate}

We now apply this to $\S(G)$. First note that the minimal right ideals of $\S(G)$ are exactly the minimal subflows of $\S(G)$. Notice also that every minimal subflow of $\S(G)$ is universal, simply by the universal property of $\S(G)$. We argue that $\M(G)$ is unique up to isomorphism, a classical theorem of Ellis \cite{ellis}. Fix $M\subseteq \S(G)$ a minimal right ideal, and let $u\in M$ be an idempotent. Suppose $\phi\colon M\to M$ is a $G$-map. Then by Fact 3, we have $\phi(p) = \phi(up) = \phi(u)p$ for any $p\in M$, hence $\phi = \lambda_{\phi(u)}|_M$. By Fact 4, we have $\phi(u) \in Mv$ for some idempotent $v$. Then since $Mv$ is a group with identity $v$, we can find $q\in M$ with $q\phi(u) = v$. Notice that $\lambda_v|_M$ is the identity on $M$ (since $M = vM$). Since $\lambda_v|_M = \lambda_q\circ \lambda_{\phi(u)}|_M$, the map $\lambda_{\phi(u)}|_M = \phi$ must be a bijection, hence a $G$-flow isomorphism. If $N$ is another minimal flow which is universal, then let $\psi\colon M\to N$ and $\theta\colon N\to M$ be $G$-maps. If $M\not\cong N$, then $\theta\circ \psi$ is not injective, contradicting that every $G$-map from $M$ to itself is an isomorphism. 

Furthermore, suppose $X$ is a minimal $G$-flow, and suppose $\phi$ and $\psi$ are two $G$-maps from $M$ to $X$. Let $u\in M$ be an idempotent, and consider $\psi^{-1}(\{\phi(u)\})\subseteq M$. If $p\in \psi^{-1}(\{\phi(u)\})$, then $\psi(pu) = \psi(p)u = \phi(u)u = \phi(uu) = \phi(u)$. It follows that $\psi\circ \lambda_p = \phi$, i.e.\ there is only one $G$-map from $M$ to $X$ up to isomorphism.

Now suppose $K$ is another topological group and that $\pi\colon G\to K$ is a continuous surjective homomorphism. We note that every $K$-flow is also a $G$-flow, where if $X$ is a $K$-flow, $x\in X$, and $g\in G$, we set $x\cdot g = x\cdot \pi(g)$. The map $\pi$ continuously extends to a map from $\S(G)$ to $\S(K)$, which we also denote by $\pi$. If $M\subseteq \S(G)$ is a minimal subflow, then $\pi[M]\subseteq \S(K)$ is also minimal and is isomorphic to $\M(K)$.
\vspace{2 mm}  

\section{Almost periodic points}
\label{Sec:APPoints}

This section proves the following key propositions which will be used in the proof of Theorems~\ref{Thm:MGMetrizable} and~\ref{Thm:GUniqueErg}. Throughout this section, we consider a Polish group $H$, which will be the same $H$ that appears in the main theorems. We fix on $H$ a compatible left-invariant metric $d$ of diameter one, and for $c> 0$, we set $U_c:= \{g\in H: d(1_H, g)< c\}$.

Given an $H$-flow $X$, the \emph{almost periodic} points of $X$, denoted $\mathrm{AP}(X)$, are those points in $X$ belonging to minimal subflows. 
\vspace{2 mm}

\begin{prop}
	\label{Prop:APClosed}
	Let $H$ be a Polish group, and suppose that $\M(H)$ is metrizable. Then for any $H$-flow $X$, the set $\mathrm{AP}(X)\subseteq X$ is closed.
\end{prop}
\vspace{2 mm}

The assumption that $\M(H)$ is metrizable is essential. Hindman and Strauss in \cite{HinStr2} show that when $H = \mathbb{Z}$ and $X = \beta \mathbb{Z}$, then $\mathrm{AP}(X)\subseteq X$ is not even Borel. In a work in preparation, Barto\v sov\'a and Zucker generalize this, showing that for any Polish group $H$ with $\M(H)$ non-metrizable and $X = S(H)$, then $\mathrm{AP}(X)\subseteq X$ is not Borel.

On $\mathrm{AP}(X)$, the relation given by $E(x,y)$ iff $x$ and $y$ belong to the same minimal subflow of $X$ is an equivalence relation, and one can ask about the complexity of this equivalence relation. It turns out that in the setting of Theorem \ref{Prop:APClosed}, this equivalence relation is as nice as possible. 
\vspace{2 mm}

\begin{prop}
	\label{Prop:APEquivClosed}
	In the setting of Proposition \ref{Prop:APClosed}, the equivalence relation $E\subseteq \mathrm{AP}(X)\times \mathrm{AP}(X)$ is closed.
\end{prop}
\vspace{2 mm}

Combining the key results from \cite{MNT} and \cite{BYMT}, we have the following.
\vspace{2 mm}

\begin{fact}
	\label{Fact:MNTandBYMT}
	Whenever $\M(H)$ is metrizable, then there is an extremely amenable, co-precompact subgroup $H^*\subseteq H$ so that $\M(H)\cong \widehat{H^*\backslash H}$, the left completion of the right coset space.
\end{fact}
\vspace{2 mm}

In particular, $\M(H)$ comes equipped with a canonical compatible metric $\partial$ inherited from the metric $d$ on $H$. The key property we need about this metric is the following.
\vspace{2 mm}

\begin{lemma}
	\label{Lem:MetricProperty}
	Let $H$ be a Polish group, and assume $\M(H)\cong \widehat{H^*\backslash H}$ is metrizable with the compatible metric $\partial$ inherited from $d$. Then whenever $\partial(p, q) < c$ and $A\ni p$ is open, we have $q\in \overline{AU_c}$.
\end{lemma}

\begin{proof}
	Fix sequences $p_n, q_n\in H$ with $H^*p_n\to p$ and $H^*q_n\to q$. We may assume for every $n< \omega$ that $d(H^*p_n, H^*q_n) < c$. By modifying $q_n$ if necessary, we may assume $p_n^{-1}q_n\in U_c$. Now if $A\ni p$ is open, then eventually $H^*p_n\in A$. Then $H^*q_n\in AU_c$, implying that $q\in \overline{AU_c}$ as desired.
\end{proof}
\vspace{2 mm}

We now assume $\M(H)$ metrizable with a compatible metric $\partial$ as in Lemma \ref{Lem:MetricProperty}, and we fix an $H$-flow $X$. Consider some collection $\{X_i: i\in I\}$ of minimal subflows of $X$; we will treat $I$ as a directed partial order. For each $i\in I$, let $\phi_i\colon \M(H)\to X_i$ be an $H$-map. The key lemma regards the right action of $\S(H)$ on $X$. In general, this action is not continuous, but the lemma states that in this setting, we recover some fragments of continuity.
\vspace{2 mm}

\begin{lemma}
	\label{Continuity}
	Suppose we have $p, q\in \M(H)$ with $\phi_i(p)\to x$ and $\phi_i(q)\to y$. Suppose $r\in \S(H)$ with $pr = q$. Then $xr = y$.
\end{lemma}

\begin{proof}
	Fix an open $B\ni y$, and fix a net $(g_j)_{j\in J}$ from $H$ with $g_j\to r$. We want to show that eventually $xg_j\in B$. Find some open $C\ni y$ and $\epsilon > 0$ with $\overline{CU_\epsilon}\subseteq B$. Eventually $\partial(pg_j, q) < \epsilon$; fix such a $g_j$. Eventually $\phi_i(q)\in C$, so by Lemma \ref{Lem:MetricProperty} for such $i\in I$ we have $pg_j\in \overline{\phi_i^{-1}(C)U_\epsilon} \subseteq \phi_i^{-1}(\overline{CU_\epsilon})$. So $\phi_i(pg_j) = \phi_i(p)g_j\in \overline{CU_\epsilon}$. As this is true for all large enough $i\in I$, we have $xg_j\in \overline{CU_\epsilon}\subseteq B$ as desired.
\end{proof}
\vspace{2 mm}

The other lemma we will need allows us to express points in $\overline{\mathrm{AP}(X)}$ as limits of certain nice nets. The proof is almost identical to that of Lemma \ref{Continuity}.
\vspace{2 mm}

\begin{lemma}
	\label{TweakNet}
	Suppose $\phi_i$ are as above, and suppose we have $p_i\in \M(H)$ with $\phi_i(p_i)\to x\in X$. If $p_i\to p\in \M(H)$, then $\phi_i(p)\to x$.
\end{lemma}

\begin{proof}
	Fix an open $B\ni x$; we want to show that eventually $\phi_i(p)\in B$. Find some open $C\ni x$ and $\epsilon > 0$ so that $\overline{CU_\epsilon}\subseteq B$. Eventually we have $\phi_i(p_i)\in C$ and $\partial(p, p_i) < \epsilon$. For such $i\in I$, since $\phi_i^{-1}(C)\ni p_i$, we have by Lemma \ref{Lem:MetricProperty} that $p\in \overline{\phi_i^{-1}(C)U_\epsilon} = \overline{\phi_i^{-1}(CU_\epsilon)}\subseteq \phi^{-1}(\overline{CU_\epsilon})$. It follows that $\phi_i(p)\in B$ as desired.
\end{proof}
\vspace{2 mm}

We can now easily complete the proof of both key propositions. First suppose $x_i\in \mathrm{AP}(X)$ with $x_i\to x$. Each $x_i$ belongs to some minimal flow $X_i$, so fix $H$-maps $\phi_i\colon \M(H)\to X_i$. Also fix $p_i\in \M(H)$ with $\phi_i(p_i) = x_i$. By passing to a subnet, we may assume $p_i\to p$, so by Lemma \ref{TweakNet}, we have $\phi_i(p)\to x$. Now fix a minimal subflow $M\subseteq \S(H)$, and consider the $H$-flow isomorphism $\lambda_p|_M\colon M\to \M(H)$. If $u\in M$ is such that $pu = p$, then by Lemma \ref{Continuity}, we have $xu = x$, i.e.\ that $x\in \lambda_x[M]$, showing that $x\in \mathrm{AP}(X)$ as desired.

For the second proposition, suppose $(x_i, y_i)\in E$ with $x_i\to x$ and $y_i\to y$. Much as above, we may assume that there are $p, q\in \M(H)$ with $\phi_i(p)\to x$ and $\phi_i(q)\to y$. Now suppose $r\in \S(H)$ with $pr = q$. By Lemma \ref{Continuity}, we have $xr = y$. It follows that $(x, y)\in E$ as desired.
\vspace{2 mm}

\section{Abstract proof of Theorems~\ref{Thm:MGMetrizable} and~\ref{Thm:GUniqueErg}}
\label{Sec:AbstractProof}

This section applies the key propositions from Section~\ref{Sec:APPoints} to prove the two main theorems from the introduction. Fix a short exact sequence $1\to H\to G\to K\to 1$ of Polish groups, and let $d$ be a compatible left-invariant metric on $G$ with diameter $1$. Then $d$ induces compatible left-invariant metrics on $H$ and $K$, which we also denote by $d$. Given $c> 0$, we set $U_c = \{g\in G: d(1_G, g) < c\}$. Then $U_c\cap H$ is the ball of radius $c$ around $1_H = 1_G$ in $H$, and $HU_c$ is the ball of radius $c$ around $1_K = H$ in $K$.

We first tackle Theorem~\ref{Thm:MGMetrizable}. The assumption that $\M(K)$ is metrizable is only used at the very end, but the assumption that $\M(H)$ is metrizable is used throughout the proof. Indeed, the proof proceeds by viewing $\M(G)$ as an $H$-flow. We write $\mathrm{AP}_H(\M(G))$ for those points in $\M(G)$ belonging to minimal $H$-subflows.
\vspace{2 mm}

\begin{lemma}
	\label{Lem:H-APDense}
	The set $\mathrm{AP}_H(\M(G))\subseteq \M(G)$ is $G$-invariant, hence dense.
\end{lemma}

\begin{proof}
	Suppose $X\subseteq \M(G)$ is a minimal $H$-subflow. Fix $g\in G$. Then $XgH = XHg = Xg$, so $Xg$ is an $H$-flow. Now suppose $y\in Xg$. Then $yHg^{-1} = yg^{-1}H\subseteq X$ is dense, so also $yH\subseteq Xg$ is dense, showing that $Xg$ is also a minimal $H$-subflow.
\end{proof}
\vspace{2 mm}

\begin{proof}[Proof of Theorem~\ref{Thm:MGMetrizable}]
	By Proposition~\ref{Prop:APClosed} and Lemma~\ref{Lem:H-APDense}, we must have $\mathrm{AP}_H(\M(G)) = \M(G)$, i.e.\ every point in $\M(G)$ belongs to a minimal $H$-subflow. Furthermore, by Proposition~\ref{Prop:APEquivClosed}, the relation $E$ defined by $E(x,y)$ iff $x\in \overline{y \cdot H}$ is a closed equivalence relation on $\M(G)$. Then $Y = \M(G)/E$ is a compact Hausdorff space and since the projection of the action of $G$ is $H$-invariant, $Y$ is a $K$-flow. This flow is minimal by minimality of the action of $G$ on $\M(G)$, hence it is metrizable and has cardinality $\mathfrak{c}$. Each equivalence class is a minimal $H$-flow, hence is metrizable and has cardinality $\mathfrak{c}$. This means that $\M(G)$ has cardinality at most $\mathfrak{c}$ and by~\cite{ZucThesis} Proposition $2.7.5$, if $\M(G)$ were non-metrizable, it would have cardinality $2^\mathfrak{c}$. 
	
	Furthermore, note that for every $y\in \M(K)$, the fiber $\pi^{-1}(\{y\})$ is an $H$-flow, giving us a map $\psi\colon Y\to \M(K)$. As $Y$ is minimal and $\M(K)$ is the universal minimal flow, we must have $\psi$ an isomorphism, i.e.\ each fiber $\pi^{-1}(\{y\})$ is a minimal $H$-flow.
\end{proof}
\vspace{2 mm}

We now turn towards the proof of Theorem~\ref{Thm:GUniqueErg}, so assume  $\M(H)$ and $\M(K)$ are metrizable and that both $H$ and $K$ are uniquely ergodic. The main idea of the proof is to apply the following measure disintegration theorem (see \cite{fur} Theorem $5.8$ and Proposition $5.9$).
\vspace{2 mm}

\begin{theorem}
	Let $X$, $Y$ be standard Borel spaces and $\phi\colon X \rightarrow Y$ a Borel map. Let $\mu\in \mathrm{P}(X)$ and $\nu = \phi_* \mu$, then there is a Borel map $y\mapsto \mu_y$ from $Y$ to $P(X)$ such that:
	\begin{itemize}
		\item[$i)$]$\mu_y(\phi^{-1}(\{y\})=1$
		\item[$ii)$]$\mu=\int \mu_y \mathrm{d}\nu(y)$.
	\end{itemize}
Moreover, if there is another such map $y\mapsto \mu_y'$, then for $\nu$-almost all $y$, $\mu_y=\mu_y'$.
\end{theorem}
\vspace{2 mm}

We apply the theorem with $X= \M(G)$, $Y = \M(G)/E = \M(K)$, and $\phi = \pi$. First note that $G$ is amenable (see the proof in Section \ref{Sec:Questions}), so let us take $\mu$ any $G$-invariant measure. Then $\nu = \phi_* \mu$ is $K$-invariant, hence it is the unique such measure. The following lemma gives us the uniqueness of the disintegration, hence the unique ergodicity of $G$ and the proof of Theorem~\ref{Thm:GUniqueErg}.
\vspace{2 mm}

\begin{lemma} 
	For $\nu$-almost all $y\in Y$, $\mu_y$ is $H$-invariant.
\end{lemma}

\begin{proof}
	We remark that by the uniqueness of the decomposition, it is easy to establish that for any countable set $(h_n)_{n\in \N}$ of elements of $H$, $\nu$-almost surely $\mu_y$ is $(h_n)_{n\in \N}$-invariant for all $n\in \N$. Since $H$ is Polish, we can assume that $(h_n)_{n\in \N}$ is dense in $H$. Since the set of $h\in H$ such that $h\cdot \mu_y = \mu_y$ is closed, it follows that for any $y\in Y$ with $\mu_y$ $(h_n)_{n\in \N}$-invariant, we in fact have that $\mu_y$ is $H$-invariant.
\end{proof}
\vspace{2 mm}

\section{Examples}
\label{Sec:Examples}

In this section, we give several examples of short exact sequences appearing in the realm of non-archimedian Polish groups. The main application of Theorems~\ref{Thm:MGMetrizable} and \ref{Thm:GUniqueErg} occurs in Subsection~\ref{Sec:Wreath}, where we discuss wreath products. Subsection~\ref{Sec:Torsion} describes instances of more general Polish group extensions, where the main theorems don't apply as clearly.

\subsection{Wreath products}
\label{Sec:Wreath}
The simplest (non trivial) setting in which short exact sequences appear is the one where we have $H$ a Polish group, $K$ is a Polish group that acts on a countable set $X$ and $G=H^X\rtimes K$. The product is defined as:
\begin{equation*}
((h_a)_{a\in X},\sigma)\cdot((g_a)_{a\in X},\tau) = ((h_a g_{\sigma(a)})_{a\in X},\sigma \tau)
\end{equation*}
This is a short exact sequence where $H^X$ is the normal subgroup and $K$ is the quotient.

We apply the main theorems to prove the following.
\vspace{2 mm}

\begin{theorem}\label{Thm:WreathProducts}
	Letting $G = H^X\rtimes K$, if $\M(H)$ and $\M(K)$ are metrizable, then so is $\M(G)$. Under those hypotheses, if $H$ and $K$ are uniquely ergodic, then so is $G$.
\end{theorem}
\vspace{2 mm}

The proof relies on the following lemma.
\vspace{2 mm}

\begin{lemma}\label{Lem:CountableProd}
Let $(G_i)_{i\in \N}$ be a family of groups such that $\M(G_i)$ is metrizable for all $i\in \N$, and set $G=\prod_{i\in \N} G_i$, then $\M(G)$ is metrizable. Moreover, if $G_i$ is uniquely ergodic for all $i\in \N$, then so is $G$.
\end{lemma}

\begin{proof}
	Using Fact~\ref{Fact:MNTandBYMT}, we know that there exists a sequence $(G_i^*)_{i\in \N}$ such that $G_i^*$ is an extremely amenable, closed, co-precompact subgroup of $G_i$.
	
	Let us consider $G^* = \prod_{i \in \N} G^*_i$ as a subgroup of $G$. It is a closed subgroup and is extremely amenable, as the property of being extremely amenable is closed under arbitrary (not just countable) products. 

	The observation $\widehat{G^* \backslash G} = \prod_{i\in \N}\widehat{ G_i^* \backslash G_i}$ gives the co-precompactness of $G^*$. Hence $\M(G)=\widehat{G^* \backslash G}$ and is metrizable.

	This also implies unique ergodicity of $G$, for let $\mu_i$ be the unique $G_i$-invariant measure on $\M(G_i)$ and $\mu$ any $G$-invariant measure on $\M(G)$. The pushfoward of $\mu$ on $\prod_{i < n} \M(G_i)$ has to be equal to $\mu_0\otimes \cdots \otimes \mu_{n-1}$ for all $n\in \N$, hence $\mu$ is uniquely determined on the basic open set of the topology of $\M(G)$ and is therefore uniquely determined.
\end{proof}
\vspace{2 mm}

\begin{proof}[Proof of Theorem \ref{Thm:WreathProducts}]
	By Theorems~\ref{Thm:MGMetrizable} and \ref{Thm:GUniqueErg}, it is enough to show that if $\M(H)$ is metrizable (and uniquely ergodic), then so is $\M(H^X)$. Lemma \ref{Lem:CountableProd} gives us exactly that.
\end{proof}
\vspace{2 mm}

Note that this result was already proven by Pawliuk and Sokic (\cite{PS} Theorem 2.1) and Sokic (\cite{Sok} Proposition 5.2) in the case where $H$ and $K$ are automorphism groups of \fr limits.
\vspace{2 mm}

\subsection{Beyond semi-direct products}
\label{Sec:Torsion}
We now consider group extensions which are not semidirect products. This subsection does not contain any particular applications of the main theorem, but is included to give some more understanding of how diverse short exact sequences of Polish groups can be. The reason that applying the main theorems is difficult here is that often, the closed subgroup $H$ is equally difficult to work with as $G$ itself.

To show that certain group extensions are not semi-direct products, we briefly discuss some of the properties of those group extensions that are. Let $H$ and $K$ be topological groups, and suppose we are given a homomorphism $\phi\colon K\to \aut{H}$. We will write $\phi(k)$ as $\phi_k$ to simplify notation. Then we can endow $H\times K$ with a group operation, where we define $(h_0, k_0)\cdot (h_1, k_1) = (h_0\phi_{k_0}(h_1), k_0k_1)$. Now suppose $\phi$ has the property that whenever $k_i\to k\in K$ and $h_i\to h\in H$, we have $\phi_{k_i}(h_i)\to \phi_k(h)\in H$. For example, when $H$ is locally compact, this property says that $\phi$ is continuous when $\aut{H}$ is given the compact-open topology. In this case, $H\times K$ endowed with the above operation is a topological group, and we denote this by $H\rtimes^\phi K$, or $H\rtimes K$ if $\phi$ is understood. Setting $G = H\rtimes K$, we identify $H$ with the closed normal subgroup $\{(h, 1_K): h\in H\}$, and the quotient $G/H$ is isomorphic to $K$, showing that $G$ is an extension of $K$ by $H$.

If $1\to H\to G\xrightarrow{\pi} K\to 1$ is a short exact sequence of topological groups, we say that the sequence \emph{splits continuously} if there is a continuous homomorphism $\alpha\colon K\to G$ so that $\pi\circ \alpha = \mathrm{id}_K$, the identity map on $K$. Such an $\alpha$ will always have closed image. When $G = H\rtimes K$, one can define $\alpha(k) = (1_H, k)$. Conversely, if $\alpha\colon K\to G$ is a continuous homomorphism with closed image and $\pi\circ \alpha = \mathrm{id}_K$, then we obtain a homomorphism $\phi\colon K\to \aut{H}$ given by $\phi_k(h) = \alpha(k)h\alpha(k^{-1})$. Then $\phi$ satisfies the required continuity property, and we have $G\cong H\rtimes^\phi K$.

\subsubsection{Ordered homogeneous metric space with distances $1$, $3$ and $4$.}

We first consider the countable homogeneous metric space with distances $1$, $3$ and $4$, which we denote by $\mathbb{F}$. This is the \fr limit of those metric spaces with distances belonging to the set $\{1, 3, 4\}$.  In $\mathbb{F}$, there are infinitely many infinite equivalence classes of points at distance $1$ and such that the distance between two non-equivalent points is $3$ or $4$ at random.

We now consider an extension $\mathbb{F}^*$ of this structure where on each equivalence class we generically put a dense linear ordering, and we leave points between different classes unordered. We set $G = \aut{\mathbb{F}^*}$.

Letting $H$ be the subgroup that stabilizes every class set-wise, then $H\subseteq G$ is a closed normal subgroup. Moreover, it is easy to prove via a back and forth that the quotient $H \backslash G$ is isomorphic to $\mathrm{S}_\infty$.

This extension cannot be a semi-direct product. To prove this we consider the element $\sigma$ of $S_\infty$ that swaps $0$ and $1$, leaving all other points fixed. We have $\sigma^2 =\mathrm{id}_\N$. Suppose now that $G$ is indeed a semidirect product; letting $\alpha\colon S_\infty\to G$ be a continuous homomorphism with $\pi\circ\alpha = \mathrm{id}_{S_\infty}$, then the element $g^* = \alpha(\sigma)$ has order $2$ in $G$ and permutes two equivalence classes, say $A$ and $B$. If we look at the action of $g^*$ on a class $C$ which $g^*$ fixes set-wise, then $g^*$ defines an automorphism of $(\mathbb{Q},<)$ of order $2$, thus it acts trivially on $C$. Now given $x\in A$, let $y = g^*(x)\in B$. Then using the homogeneity of $\mathbb{F}^*$, we can find $z\in C$ with $d(x, z) = 3$ and $d(y, z) = 4$. Therefore we must have $g^*(z) \neq z$, a contradiction.

\subsubsection{The switching group} 

We now consider the structure $\mathbb{F}$ formed by first considering the Rado graph $\mathbb{H}=(V,E)$. The structure $\mathbb{F}$ has domain $V$ and the language only has one $4$-ary relation symbol $R$ with the following condition:
\begin{equation*}
R^\mathbb{F}(x,y,w,z) \Leftrightarrow \left((E^\mathbb{H}(x,y) \wedge E^\mathbb{H}(w,z))\vee (\lnot E^\mathbb{H}(x,y) \wedge \lnot E^\mathbb{H}(w,z) \right)
\end{equation*}
where $x,y,w,z$ are vertices. We obviously have $\aut{\mathbb{H}}\vartriangleleft \aut{\mathbb{F}}$. The quotient is $\mathbb{Z}/2\mathbb{Z}$.

Again, this is not a semi-direct product, otherwise we would have $f$ an involution of the vertices such that $E(x,y) \Leftrightarrow \lnot E (f(x),f(y))$, which is impossible because we would have $E(x,f(x))$, $\lnot E((f(x),f^2(x))$, and $f^2(x) = x$.

\subsubsection{Partitioned $(\mathbb{Q},<)$}
We partition $\mathbb{Q}$ in dense codense classes that we name $(E_i)_{i \in \N}$. We define an equivalence relation $E$ on $\mathbb{Q}$:
\begin{equation*}
E(x,y) \Leftrightarrow \exists i\in \N \ : \ (x,y)\in E_i.
\end{equation*}
We let $G$ be the subgroup of $\aut{\mathbb{Q},<}$ fixing the equivalence relation $E$, and we let $H$ be the subgroup of $G$ that fixes each $E_i$ setwise; it is normal in $G$.

Again, a torsion argument allows us to prove this is not a semi direct product.

\section{Combinatorial proof of Theorem~\ref{Thm:MGMetrizable}}	
\label{Sec:ComboProof}

This section provides a combinatorial proof of a weakening of Theorem~\ref{Thm:MGMetrizable}; this proof does not show that each fiber is a minimal $H$-flow. The advantage of this proof is that it is ``quantitative'' in a sense that will be made precise. We will first reprove the theorem in the case that $G$ is non-Archimedean, and then discuss the general case.

\subsection{The non-Archimedean case}
\label{SubSec:NonArch}

We first assume that $G$, hence also $H$ and $K$, are non-Archimedean. So fix $\{U_n: n < \omega\}$ a base at $1_G$ of clopen subgroups. Then $\{H\cap U_n: n < \omega\}$ and $\{\pi[U_n]: n < \omega\}$ are bases of clopen subgroups at the identity in $H$ and $K$, respectively. For instance, if $G = \aut{\mathbf{K}}$ for some \fr structure $\mathbf{K} = \flim{\mathcal{K}}$, and we write $\mathbf{K} = \bigcup_n \mathbf{A}_n$ as an increasing union of finite substructures, then we can let $U_n$ be the pointwise stabilizer of $\mathbf{A}_n$. We will need the following definition, which we translate from the \fr setting to the group setting.

We consider the left coset space $G_n := G/U_n$, which is countable. When $U_n = \mathrm{Stab}(\mathbf{A}_n)$, then $G_n$ can be identified with the set $\emb{\mathbf{A}_n, \mathbf{K}}$. The group $G$ acts on $G_n$ on the left in the natural way. If $X$ is a compact space, then $X^{G_n}$ becomes a right $G$-flow, where for $\gamma\in X^{G_n}$ and $g_0\in G$, we define $\gamma\cdot g_0$ via $\gamma\cdot g_0(g_1U_n) := \gamma(g_0g_1U_n)$.
\vspace{2 mm}

\begin{defin}
	\label{Def:RamseyDegree}
	Fix $n, m < \omega$. We say that the clopen subgroup $U_n\subseteq G$ has \emph{Ramsey degree} $m < \omega$ if the following both hold.
	\begin{enumerate}
		\item 
		For any $r< \omega$ and any coloring $\gamma\colon G_n\to r$, there is some $\delta\in \overline{\gamma\cdot G}$ and some $F\subseteq r$ with $\delta\in F^{G_n}$ and $|F|\leq m$. Equivalently, for any $\gamma$ as above, there is $p\in \S(G)$ with $|\im{\gamma\cdot p}|\leq m$.
		\item 
		There is a surjective coloring $\gamma: G_n\to m$ so that $\overline{\gamma\cdot G}$ is a minimal $G$-flow. We often call such colorings \emph{minimal}, or \emph{$G$-minimal} to emphasize the group.
	\end{enumerate}
\end{defin}
\vspace{2 mm}

When $U_n = \mathrm{Stab}(\mathbf{A}_n)$, then Definition~\ref{Def:RamseyDegree} coincides with the Ramsey degree (for embeddings) of $\mathbf{A}_n\in \mathcal{K}$. We then have the following theorem.
\vspace{2 mm}

\begin{theorem}[\cite{ZucMetr}]
	\label{Thm:RamseyDegMetr}
	$\M(G)$ is metrizable iff for every $n< \omega$, the subgroup $U_n$ has finite Ramsey degree.
\end{theorem}
\vspace{2 mm}

By Theorem~\ref{Thm:RamseyDegMetr} and our assumption that $\M(H)$ and $\M(K)$ are metrizable, we know that for every $n< \omega$, the subgroups $H\cap U_n\subseteq H$ and $\pi[U_n]\subseteq K$ have finite Ramsey degrees $m_H$ and $m_K$, respectively. We will use these to bound the Ramsey degree of $U_n$. The following proposition will prove Theorem~\ref{Thm:MGMetrizable} in the non-Archimedean case.
\vspace{2 mm}

\begin{prop}
	\label{Prop:RamseyDegBound}
	With $m_H$ and $m_K$ as above, the Ramsey degree $m_G$ of $U_n$ satisfies $m_G\leq m_H\cdot m_K$.
\end{prop} 

\begin{proof}
	Write $H_n := H/(H\cap U_n)$ and $K_n := K/\pi[U_n]$. Then we have an inclusion map $H_n\hookrightarrow G_n$ as well as a projection map $\pi_n: G_n\to K_n$, both of which respect the various left actions. Furthermore, the equivalence relation $E_n$ induced by $\pi_n$ is exactly the orbit equivalence relation of the left $H$-action on $G_n$. From now on, we will view $H_n$ as a subset of $G_n$. 
	
	Let $\gamma\colon G_n\to r$ be a coloring. Find $\{g_k: k< \omega\}\subseteq G$ so that $\{g_k\cdot H_n: k< 
	\omega\}$ lists the $E_n$-classes in $G_n$. We now inductively define a sequence of colorings $\{\gamma_k: k< \omega\}\subseteq r^{G_n}$. Set $\gamma = \gamma_0$. If $\gamma_k$ is defined for some $k < \omega$, consider $\gamma_k\cdot g_k|_{H_n}$. We can find $p_k\in \S(H)\subseteq \S(G)$ so that $|\gamma_k\cdot g_k\cdot p_k[H_n]|\leq m_H$. Then set $\gamma_{k+1} = \gamma_k\cdot g_k\cdot p_k\cdot g_k^{-1}$. Note that $|\gamma_{k+1}[g_kH_n]|\leq m_H$. Also notice that $g_k\cdot p_k\cdot g_k^{-1}\in S(H)$, implying that for any $i\leq k$, we have $\gamma_{k+1}[g_iH_n]\subseteq \gamma_k[g_iH_n]$. Let $\delta\colon G_n\to r$ be any cluster point of $\{\gamma_k: k< \omega\}$. Then for each $k< \omega$, $\delta[g_kH_n]$ is a subset of $r$ of size at most $m_H$. This allows us to produce a finite coloring $\eta\colon K_n\to [m_H]^{\leq r}$, and we can find $q\in S(G)$ with $|\eta\cdot \pi(q)[K_n]|\leq m_K$. It follows that $\delta\cdot q[G_n]\leq m_H\cdot m_K$ as desired.
\end{proof}
\vspace{2 mm}

\subsection{The general case}

In the general case, we will need the following analogue of Theorem~\ref{Thm:RamseyDegMetr}. If $G$ is a Polish group equipped with a compatible left-invariant metric of diameter $1$ and $X$ is a compact metric space, then $\mathrm{Lip}_G(X)$ will denote the space of $1$-Lipschitz functions from $G$ to $X$. When endowed with the topology of pointwise convergence, $\mathrm{Lip}_G(X)$ becomes a compact space. We have a right action of $G$ on $\mathrm{Lip}_G(X)$; if $f\in \mathrm{Lip}_G(X)$ and $g\in G$, then $f\cdot g$ is given by $f\cdot g(h) = f(gh)$. This action is continuous, turning $\mathrm{Lip}_G(X)$ into a $G$-flow.
\vspace{2 mm}

\begin{theorem}
	\label{Thm:MGMetrGeneral}
	$\M(G)$ is metrizable iff for every $c> 0$, there is $k< \omega$ so that if $X$ is a compact metric space and $f\in \mathrm{Lip}_G(X)$, there is $\phi\in \overline{f\cdot G}$ so that $\overline{\phi[G]}\subseteq X$ can be covered by $k$-many balls of radius $c$.
\end{theorem}

\begin{rem}
	One should think of $k < \omega$ above as the ``Ramsey degree'' of the constant $c > 0$, and we will use this terminology.
\end{rem}

\begin{proof}
	Assume first that $M(G)$ is metrizable. Fact~\ref{Fact:MNTandBYMT} gives $\M(G)$ a canonical compatible metric $\partial$ which satisfies the following property: for any $G$-map $\psi\colon \M(G)\to \S(G)$, for any compact metric space $X$, and for any $1$-Lipschitz function $f\colon G\to X$, then upon extending $f$ to $\S(G)$, the function $f\circ \psi$ is $1$-Lipschitz for $\partial$.

	For any function $\phi\in \overline{f\cdot G}$ with $\overline{\phi\cdot G}$ minimal, there is a minimal subflow $M\subseteq \S(G)$ and $p\in M$ with $\phi = f\cdot p$. It follows that $\overline{\phi[G]} = f[M]$. Given $c > 0$, pick $k< \omega$ so that $\M(G)$ can be covered by $k$-many $\partial$-balls of radius $c$. Then if $\psi\colon \M(G)\to M$ is an isomorphism, $f\circ \psi$ is $1$-Lipschitz and the result follows.
	
	For the other direction, suppose $\M(G)$ is not metrizable. Fixing minimal $M\subseteq \S(G)$ and mimicking the proof of Lemma 2.5 in \cite{BYMT}, we can find a constant $c > 0$, $\{p_n: n< \omega\}\subseteq M$, and functions $\{f_n: n< \omega\}\subseteq \mathrm{Lip}_G([0,1])$. with $f_n(p_n) = c$ and $f_m(p_n) = 0$ whenever $m\neq n$. Then for any $N< \omega$, we can form $f\colon G\to [0,1]^N$ given by $f = (f_n)_{n< N}$. If $p\in M$, then $f\cdot p\in \mathrm{Lip}_G([0,1]^N)$ has minimal orbit closure, but covering $f\cdot p[G]$ requires at least $N$-many balls of radius $c$, and $N< \omega$ is arbitrary. 
\end{proof}
\vspace{2 mm}

We now return to the setting where $1\to H\to G\to K\to 1$ is a short exact sequence of Polish groups, and we assume that $\M(H)$ and $\M(K)$ are metrizable. Therefore given $c> 0$, we let $m_H, m_K < \omega$ be the Ramsey degrees of $c$ for $H$ and $K$, respectively.
\vspace{2 mm}

\begin{prop}
	\label{Prop:RamseyBoundGen}
	Suppose $c > 0$ and $m_H, m_K< \omega$ are as above. Then letting $m_G$ denote the Ramsey degree of $2c$ in $G$, we have $m_G \leq m_H\cdot m_K$.
\end{prop}

\begin{proof}
	Let $X$ be a compact metric space, and let $\gamma\in \mathrm{Lip}_G(X)$. Find group elements $\{g_i: i< [G:H]\}$ so that $\{g_iH: i < [G:H]\}$ lists the elements of $K$. We proceed much as in the non-Archimedean case. Set $\gamma = \gamma_0$. If $\gamma_\alpha$ is defined for some ordinal $\alpha < [G:H]$, find $p_\alpha\in S(H)$ with the property that $\overline{\gamma_\alpha\cdot g_\alpha\cdot p_\alpha[H]}\subseteq X$ can be covered by at most $m_H$-many balls of radius $c$. Then set $\gamma_{\alpha+1} = \gamma_\alpha\cdot g_\alpha\cdot p_\alpha\cdot g_\alpha^{-1}$. If $\gamma_\beta$ has been defined for each $\beta < \alpha$ and $\alpha$ is a limit ordinal, let $\gamma_\alpha$ be any cluster point of $\{\gamma_\beta: \beta < \alpha\}$. 
	
	Letting $\kappa = [G:H]$ (so $\kappa = \omega$ or $\mathfrak{c}$), notice that for any $\alpha < \kappa$, we have $\overline{\gamma_\kappa[g_\alpha H]}\subseteq \overline{\gamma_{\alpha+1}[g_\alpha H]}$. Now form $K(X)$, the space of compact subsets of $X$ equipped with the Hausdorff metric. Notice that since $\gamma_\kappa$ is $1$-Lipschitz, the function $\eta\colon K\to K(X)$ given by $\eta(g_\alpha H) = \overline{\gamma_\kappa[g_\alpha H]}$ is $1$-Lipschitz. Find $q\in S(G)$ so that $\overline{\eta\cdot q[K]}\subseteq K(X)$ can be covered by at most $m_K$-many balls of radius $c$. It follows that $\overline{\gamma_\kappa\cdot q[G]}\subseteq X$ can be covered by at most $m_H\cdot m_K$ balls of radius $2c$.
\end{proof}
\vspace{2 mm}

\section{Questions}
\label{Sec:Questions}

\subsection*{Analyzing the fibers}
Our first question is a strengthening of Theorem \ref{Thm:MGMetrGeneral}, where we aim to describe the minimal $H$-flows of the form $\pi^{-1}(\{y\})$ for $y\in \M(K)$.
\vspace{2 mm}

\begin{que}
	\label{Que:Fibers}
	Let $1\to H\to G\xrightarrow{\pi} K\to 1$ be a short exact sequence of Polish groups with $\M(H)$ and $\M(K)$ metrizable. Is it the case that $\pi^{-1}(\{y\})\cong \M(H)$ for every $y\in \M(K)$?
\end{que} 
\vspace{2 mm}

In the interest of understanding this question better, we focus on the case that $G$ is non-Archimedean, and we continue to use notation from Section~\ref{SubSec:NonArch}. For each $n< \omega$, let $m_H(n)$, $m_G(n)$, and $m_K(n)$ denote the Ramsey degrees of the clopen subgroups $H\cap U_n$, $U_n$, and $\pi[U_n]$ in $H$, $G$, and $K$, respectively. In Proposition~\ref{Prop:RamseyDegBound}, we showed that $m_G(n)\leq m_H(n)\cdot m_K(n)$ for every $n<\omega$.
\vspace{2 mm}

\begin{prop}
	\label{Prop:RamseyProductFibers}
		Suppose for each $n< \omega$, we have $m_G(n) = m_H(n)\cdot m_K(n)$. Then for any $y\in \M(K)$, we have $\pi^{-1}(\{y\})\cong \M(H)$.
\end{prop}
\vspace{2 mm}

We will need the following lemma.
\vspace{2 mm}

\begin{lemma}
	\label{Lem:MGRepresentation}
	Suppose for each $n< \omega$ that $\gamma_n\colon G_n\to m_G(n)$ is a surjective coloring with $\overline{\gamma_n\cdot G}$ minimal. Form $\gamma = (\gamma_n)_{n< \omega}\in \prod m_G(n)^{G_n}$, and assume also that $\overline{\gamma\cdot G}$ is minimal. Then $\M(G)\cong \overline{\gamma\cdot G}$. 
\end{lemma}

\begin{rem}
	Notice that we may then think of $\M(G)$ as the space of all such $\gamma = (\gamma_n)_{n< \omega}$.
\end{rem}

\begin{proof}
	We use some of the ideas from~\cite{ZucMetr}. Recall that we have $\S(G) = \varprojlim \beta G_n$. We let $\pi_n$ denote the projection onto coordinate $n$. Explicitly, given $p\in \S(G)$, we have $\pi_n(p) = \displaystyle\lim_{g\to p} gU_n\in \beta G_n$. Notice that if $gU_n = hU_n$, then $\pi_n(pg) = \pi_n(ph)$. If $x\in \beta G_n$, we set $p\cdot x = \displaystyle\lim_{gU_n\to x} \pi_n(pg)\in \beta G_n$. If $p, q\in \S(G)$, we note that $\pi_n(pq) = p\cdot \pi_n(q)$.

	Let $M\subseteq \S(G)$ be a minimal subflow. It is shown in \cite{ZucMetr} that $|\pi_n[M]| = m_G(n)$. The map $\lambda_\gamma\colon M\to \overline{\gamma\cdot G}$ given by $\lambda_\gamma(p) = \gamma\cdot p$ is a $G$-map, and it suffices to show that it is injective. Notice first that $\gamma\cdot p = (\gamma_n\cdot p)_{n< \omega}$ for any $p\in \S(G)$.  Suppose $u\in M$ is an idempotent such that $\gamma = \gamma\cdot u$. For each $n<\omega$, consider the coloring $\lambda_u^n\colon G_n\to \pi_n[M]$ given by $\lambda_u^n(gU_n) = \pi_n(ug)$. Then $\lambda_u^n$ and $\gamma_n\cdot u = \gamma_n$ are both surjective, minimal colorings each taking $m_G(n)$ values. Because $m_G(n)$ is the Ramsey degree of $U_n$ in $G$, we must have that $\lambda_u^n$ and $\gamma_n$ are equivalent, i.e.\ for $g_0U_n, g_1U_n\in G_n$, we have $\pi_n(ug_0) = \pi_n(ug_1)$ iff $\gamma_n(g_0U_n) = \gamma_n(g_1U_n)$. 
	
	So suppose $\gamma\cdot p = \gamma\cdot q$ for some $p, q\in M$. In particular, upon extending the coloring to $\beta G_n$, we have $\gamma_n(\pi_n(p)) = \gamma_n(\pi_n(q))$. So also we have $u\cdot \pi_n(p) = u\cdot \pi_n(q)$. But since $u$ is a left identity for $M$, we have $\pi_n(p) = \pi_n(q)$. Since $n< \omega$ is arbitrary, we have $p = q$ as desired.
\end{proof}
\vspace{1 mm}

\begin{proof}[Proof of Proposition~\ref{Prop:RamseyProductFibers}]
	Suppose $m_G(n) = m_H(n)\cdot m_K(n)$. Let $\gamma = (\gamma_n)_{n<\omega}$ be as in Lemma~\ref{Lem:MGRepresentation}, so that $\M(G)\cong \overline{\gamma\cdot G}$. It is enough to show that $\overline{\gamma\cdot H}\cong \M(H)$. From the proof of Theorem~\ref{Thm:MGMetrizable}, we know that $\overline{\gamma\cdot H}$ is minimal, so also each $\overline{\gamma_n\cdot H}$ is minimal. These in turn factor onto $\overline{\gamma_n|_{H_n}\cdot H}$, so also $\overline{(\gamma_n|_{H_n})_{n<\omega}\cdot H}$ is minimal. We then note that since $\gamma_n$ was a surjective, minimal $m_G(n)$-coloring, we must have that $\gamma_n|_{H_n}$ is a surjective, $H$-minimal $m_H(n)$-coloring. Therefore we are done by Lemma~\ref{Lem:MGRepresentation}.
\end{proof}

\subsection*{Unique ergodicity and open subgroups}

Theorem \ref{Thm:GUniqueErg} tells us that metrizability of the universal minimal flow and unique ergodicity are stable under group extension. Therefore we can add entries to the following table, describing which dynamical properties are preserved under which group-theoretic operations.
\vspace{4 mm}

\begin{tabular}{l| c c c c}
{} & Amenability & Ext. amen. & Metriz. of the UMF & + unique ergo. \\
\hline
Group extensition & \checkmark & \checkmark & \checkmark & \checkmark \\
Countable products & \checkmark & \checkmark & \checkmark & \checkmark \\
Direct limits & \checkmark & \checkmark & $\times$ & $\times$ \\
Open subgroup & \checkmark & \checkmark & \checkmark & ?
\end{tabular}
\vspace{3 mm}

The arguments for open subgroups can be found in the proof of Lemma $13$ in \cite{BPT}. We already proved that metrizability of the universal minimal flow and unique ergodicity are stable under group extension and countable products, thus we only need to prove that amenability and extreme amenability are stable under group extension and direct limits. We will also produce counter examples for the failure of stability of metrizability of the universal minimal flow and unique ergodicity under direct limits.

We will use the following characterization of amenability (see \cite{Pest18}):
\begin{defin}A Polish group $G$ is \emph{amenable} if every continuous affine action of $G$ on a compact convex subspace of a locally convex topological vector space admits a fixed point.
\end{defin}
As this definition is very close to extreme amenability, we will only produce the proof of stability of amenability, the proof for extreme amenability following the same steps.
\vspace{2 mm}

\emph{Stability under group extension:} We consider the short exact sequence
\begin{equation*}
1\to H\to G\xrightarrow{\pi} K\to 1,
\end{equation*}
where $H$ and $K$ are amenable, and an affine continuous action $G\curvearrowright X$ on a convex compact space. $H$ acts on $X$ and therefore admits fixed points. Since the action is affine, the set of fixed points is convex, it is closed by continuity of the action. Morevover, $K$ acts on this set of fixed point, and by amenability this action also admits a fixed point. This point will then be $G$ invariant. 
\vspace{2 mm}

\emph{Stability under direct limits:} A group $G$ is a \emph{direct limit} of the sequence $(G_n)_{n< \omega}$ if we have
\begin{equation*}
G_0 \leq \cdots \leq G_n \leq \cdots \leq G
\end{equation*}
and $\overline{\bigcup G_n}$ is dense in $G$. We are interrested in the case where $G_n$ is amenable for all $n\in \N$.

Again, consider an affine continuous action $G\curvearrowright X$ on a convex compact space. $G_n\curvearrowright X$ admits a fixed point $x_n$. Since $X$ is compact, $(x_n)_{n\in \N}$ admits a cluster point $x\in X$. Then $x$ is $G_n$-invariant for all $n\in  \N$, hence it is $G$-invariant.
\vspace{2 mm}

\emph{Counterexamples:} In \cite{KPT} appendix $2$, it is proven that non compact locally compact group have non metrizable universal minimal flows. Moreover, Weiss proved in \cite{Weiss} that discrete group are never uniquely ergodic. Hence, any locally finite discrete group produce the counterexample we need, for instance the group of permutations of $\N$ with finite support.
\vspace{2 mm}

The next question concerns the question mark appearing in the array.
\vspace{2 mm}

\begin{que}
Let $G$ be a uniquely ergodic Polish group with metrizable universal minimal flow and $U$ an open subgroup. Is $U$ uniquely ergodic?
\end{que}
\vspace{2 mm}

Another question concerns the connection between $\M(G)$ metrizable and unique ergodicity. For all known examples of  uniquely ergodic Polish groups $G$, we have that $\M(G)$ is metrizable.
\vspace{2 mm} 

\begin{que}
Let $G$ be a uniquely ergodic Polish group. Is $\M(G)$ metrizable?
\end{que}
\vspace{2 mm}





\vspace{5 mm}

\noindent
Colin Jahel

\noindent
Universit\'e Paris Diderot, IMJ-PRG

\noindent
colin.jahel@imj-prg.fr
\vspace{5 mm}

\noindent
Andy Zucker

\noindent
Universit\'e Paris Diderot, IMJ-PRG

\noindent
andrew.zucker@imj-prg.fr


\begin{thebibliography}{99}
	\bibitem{AKL}
	O. Angel, A. S. Kechris, and R. Lyons. Random orderings and unique ergodicity of automorphism
	groups. \emph{J. European Math. Society}, \textbf{16} (2014), 2059--2095.
	
	\bibitem{BYMT} 
	I. Ben Yaacov, J. Melleray, and T. Tsankov, Metrizable universal minimal flows of Polish groups have a comeager orbit, \emph{Geom.\ Func.\ Anal.}, \textbf{27(1)} (2017), 67--77.
	
    \bibitem{BPT}
    M. Bodirsky, M. Pinsker, T. Tsankov, Decidability of definability, \emph{J. Symbolic Logic} \textbf{78} (2013), no. 4, 1036–1054.
    
    \bibitem{ellis}
    R. Ellis, Universal minimal sets, \emph{Proc. Amer. Math. Soc.}, \textbf{11} (1960), 540--543.	
    
    \bibitem{fur}
 H. Furstenberg, Recurrence in ergodic theory and combinatorial number theory, \emph{Princeton University Press, Princeton, N.J.}, 1981.
	
	\bibitem{HinStr} N.\ Hindman and D.\ Strauss, \emph{Algebra in the Stone-\v{C}ech Compactification}, 2nd Edition, De Gruyter, 2012.
	
	\bibitem{HinStr2} N.\ Hindman and D.\ Strauss, Topological properties of some algebraically defined subsets of $\beta \mathbb{N}$. \emph{Topol. Appl.} \textbf{220} (2017), 43–49.



	
	
	\bibitem{KPT} A.\ Kechris, V.\ Pestov, and S.\ Todor\v{c}evi\'c, Fra\"{i}ss\'{e} limits, Ramsey theory, and topological dynamics of automorphism groups, \emph{Geometric and Functional Analysis}, \textbf{15} (2005), 106--189.
	
	\bibitem{MNT}
	J. Melleray, L. Nguyen Van Th\'e, and T. Tsankov, Polish groups with metrizable universal minimal flow, \emph{Int.\ Math.\ Res.\ Not.\ IMRN}, no.\ 5 (2016), 1285--1307.
	
	\bibitem{PS} 
	M. {Pawliuk} and M. {Soki\'c},
    {Amenability and unique ergodicity of automorphism groups of countable homogeneous directed graphs},
  \emph{Ergodic Theory and Dynamical Systems}, (to appear).
 
	
	\bibitem{Pest18}
	 Vladimir G. Pestov,
     Amenability versus property {$(T)$} for non-locally compact
              topological groups,
   \emph{Trans. Amer. Math. Soc.},
   no.\ 10, (2018), 7417--7436.
   
   \bibitem{Sok}
   M. Sokic, Relational quotients, \emph{Fund.\ Math.}, \textbf{221}, (2013).
   
   \bibitem{Weiss}
   B.\ Weiss, Minimal models for free actions, \emph{Contemporary Mathematics}, \textbf{567} (2012), 249--264.
	
	\bibitem{ZucMetr} A.\ Zucker, Topological dynamics of automorphism groups, ultrafilter combinatorics, and the Generic Point Problem, \emph{Transactions of the American Mathematical Society}, \textbf{368(9)}, (2016).
	
	\bibitem{ZucThesis} A.\ Zucker, \emph{New directions in the abstract topological dynamics of Polish groups}, PhD.\ thesis, Carnegie Mellon University, April 2018.
\end{thebibliography}
\end{document}